\documentclass[11pt,reqno]{amsart}
\usepackage{latexsym,amssymb,graphicx,multicol,mathrsfs,color,xypic,amscd, amsxtra, verbatim}

\theoremstyle{plain}

\newtheorem{theorem}{Theorem}[section]
\newtheorem{proposition}[theorem]{Proposition}
\newtheorem{lemma}[theorem]{Lemma}
\newtheorem{corollary}[theorem]{Corollary}
\newtheorem*{mainresult}{Main Result}
\newtheorem*{maincorollary}{Corollary}

\theoremstyle{definition}
\newtheorem{remark}{Remark}

\newtheorem{definition}[theorem]{Definition}

\begin{document}
\newcommand{\m}{\ensuremath{m}}
\newcommand{\Exc}{\ensuremath{\text{Exc}}}
\def\ra{\rightarrow}
\def\nb{\nobreakdash}
\def\SM{\overline{\mathcal{M}}}
\def\sigman{\{\sigma_j\}_{j=1}^{n}}
\def\taum{\{\tau_j\}_{j=1}^{m}}
\def\NSM{\widetilde{\mathcal{M}}}
\def\NM{\widetilde{M}}
\def\A{\mathcal{A}}
\def\M{\overline{M}}
\def\N{\text{N}}
\def\C{\mathcal{C}}
\def\Q{\mathbb{Q}}
\def\P{\mathbb{P}}
\def\Exc{\text{Exc\,}}
\def\Sing{\text{Sing\,}}
\def\Proj{\text{Proj\,}}
\def\Pic{\text{Pic\,}}
\def\deg{\text{deg}\,}
\def\Z{\mathbb{Z}}
\def\H{\mathbf{H}}

\title[Ample divisors on $\M_{0,\A}$]{Ample divisors on $\M_{0,\A}$\\ (with applications to log MMP for $\M_{0,n}$)}
\author{Maksym Fedorchuk, David Ishii Smyth}
\maketitle

\maketitle
\begin{abstract}
We introduce a new technique for proving positivity of certain divisor classes on $\M_{0,n}$ and its weighted variants $\M_{0,\A}$. Our methods give an unconditional description of the symmetric weighted spaces $\M_{0,\A}$ as log canonical models of $\M_{0,n}$.
\end{abstract}

\tableofcontents
\pagebreak
\section{Introduction}
In \cite{Hassett2}, Brendan Hassett initiated the problem of studying certain log canonical models of moduli spaces of curves. For any rational number $\alpha$ such that $K_{\SM_{g,n}}+\alpha\Delta$ is an effective divisor on the moduli stack of $n$\nb-pointed genus $g$ curves, we may define
\begin{align*}
\M_{g,n}(\alpha)= \Proj \oplus_{m \geq 0} H^0(\SM_{g,n}, m (K_{\SM_{g,n}}+\alpha\Delta)),
\end{align*}
where the sum is taken over $m$ sufficiently divisible, and ask whether the spaces $\M_{g,n}(\alpha)$ admit a modular description. In the case $g=0$, it is easy to see that $K_{\M_{0,n}}+\alpha \Delta$ is effective if and only if $\alpha>\frac{2}{n-1}$, and Matthew Simpson has described the corresponding models, assuming the $S_{n}$-equivariant $F$-conjecture \cite{Simpson}.

\begin{theorem}[Simpson]\label{T:Simpson} 
Assume that the $S_{n}$-equivariant F-conjecture holds.
\begin{itemize}
\item[1.] If $\alpha \in \Q \cap (\frac{2}{k+2},\frac{2}{k+1}]$ for some $k=1, \ldots, \lfloor \frac{n-1}{2} \rfloor$, then $\M_{0,n}(\alpha) \simeq \M_{0,\A}$,\\
the moduli space of $\A$-stable curves, with $\A=\{\underbrace{1/k, \ldots, 1/k}_{n}\}$.
\item[2.] If $\alpha \in \Q \cap(\frac{2}{n-1},\frac{2}{ \lfloor n/2 \rfloor+1}]$, then $\M_{0,n}(\alpha)=(\P^{1})^n//\text{SL}_{2}$.
\end{itemize}
\end{theorem}
In addition, when $k> \lceil \frac{n}{3} \rceil$, Simpson has given an unconditional proof of Theorem \ref{T:Simpson} by constructing the corresponding spaces $\M_{0,\A}$ as inverse limits of GIT quotients.
\begin{theorem}[Simpson]\label{T:Simpson2}
If $k>\lceil \frac{n}{3} \rceil$, the conclusions of Theorem \ref{T:Simpson} hold without assuming the $S_{n}$-equivariant $F$-conjecture.
\end{theorem}
The purpose of this paper is to give an unconditional proof of Theorem \ref{T:Simpson} which is valid for all $k$, thus completing Hassett's proposed log minimal model program for $\M_{0,n}$. Our methods are quite different from Simpson's in that they produce ample divisors independent of any input from geometric invariant theory or Kollar's results on positivity of push-forwards of dualizing sheaves \cite{Kollar}. Our methods are applicable without knowing \emph{a priori} that $\M_{0,\A}$ is projective, and can thus be viewed as an elementary, characteristic-indepenent proof of the projectivity of $\M_{0,\A}$, as well as the finite-generation of all log-canonical section rings $R(\M_{0,n},K_{\M_{0,n}}+\alpha\Delta)$. In forthcoming work \cite{Smyth2}, the second author will use similarly direct methods to verify the projectivity of $\M_{1,n}(m)$, the moduli space of genus-one $m$-stable curves \cite{Smyth1}, for which no other proof of projectivity is currently known. Furthermore, the ample divisors produced by explicit one-parameter intersection theory will immediately yield interpretations of certain spaces $\M_{1,n}(m)$ as log-canonical models of $\M_{1,n}$, in the same spirit as Theorem 1.1.

The key ingredient in our argument is a new method for verifying the positivity of certain linear combinations of tautological divisor classes on $\M_{0,\A}$. To get a feel for the method, let us consider the problem: For which values of $c \in \Q$ is the divisor $c\psi-\Delta$ nef on $\M_{0,n}$?

Given a generically smooth $n$-pointed stable curve $(\C \rightarrow B, \{\sigma_i\}_{i=1}^{n})$ over a smooth curve $B$, there exists a sequence of elementary blow-downs:
\[
\C_{0} \rightarrow \C_{1} \rightarrow \cdots \C_{N-1} \rightarrow \C_{N},
\]
where $\C_{0}$ is the minimal desingularization of the total space $\C$, and $\C_{N}$ is a smooth $\P^{1}$-bundle over $B$. Let $\{\sigma^{i}_j\}_{j=1}^{n}$ denote the sections of $\pi_i: \C_{i} \rightarrow B$ obtained as the images of $\sigman$ on $\C_{i}$, and consider the function $G:[0,N] \rightarrow \Q$ defined by
$$
G_c(i):=-\frac{c}{n-1}\sum_{1 \leq j < k \leq n}(\sigma^i_j-\sigma^i_k)^2-\deg(\pi_{i \, *}(\Sing(\pi_i))).
$$
By expanding out the product, one verifies immediately that $G_c(0)=c(\psi.B)-(\Delta.B)$. On the other hand, $G_c(N)=0$ since $\pi_N$ is a smooth $\P^{1}$-bundle, and the difference of any two sections on a $\P^{1}$-bundle is numerically equivalent to a collection of fibers. Thus, if we choose $c \in \Q$ so that the function $G_c(i)$ is decreasing, we conclude that $c(\psi.B)-(\Delta.B)$ is positive.

It is not hard to see that if the exceptional divisor of $\C_{i} \rightarrow \C_{i+1}$ contracts $r$ sections, then
\[
G_c(i)-G_c(i+1)=\frac{cr(n-r)}{n-1}-1.
\]
The hypothesis that $(\C \rightarrow B, \{\sigma_i\}_{i=1}^{n})$ is stable implies that each blow-down contracts $r \geq 2$ sections, and we conclude that 
\[ c \geq \frac{n-1}{2(n-2)} \implies c(\psi.B)-(\Delta.B) \geq 0\]
for any 1-parameter family of stable curves with smooth general fiber. In fact, one can see this inequality in a more direct fashion, simply by rewriting $c\psi-\Delta$ as an effective sum of boundary divisors. The advantage to our approach is that it generalizes directly to the weighted spaces $\M_{0,\A}$. Using the inductive description of the boundary of $\M_{0,\A}$, as well as more sophisticated `sums-of-squares' functions, we will prove positivity statements for various divisors on $\M_{0,\A}$.

In order to understand which divisors are relevant to the log minimal model program for $\M_{0,n}$, let us recall Simpson's proof of Theorem \ref{T:Simpson}: Fixing $\alpha \in \Q \cap (\frac{2}{k+2},\frac{2}{k+1}]$, he considers the divisor $K_{\M_{0,n}}+\alpha\Delta$ and the birational contraction $$\phi: \M_{0,n} \rightarrow \M_{0,\A},\,\,\,\,\,\, \A=\{\underbrace{1/k, \ldots, 1/k}_{n}\}.$$ He shows that
\begin{itemize}
\item[(1)]$(K_{\M_{0,n}}+\alpha\Delta)-\phi^*\phi_*(K_{\M_{0,n}}+\alpha\Delta)$ is effective.
\item[(2)]If the $S_{n}$-equivariant $F$-conjecture holds, then $\phi_*(K_{\M_{0,n}}+\alpha\Delta)$ is ample.
\end{itemize}
Together, (1) and (2) immediately imply the statement of the theorem. The ampleness of $\phi_*(K_{\M_{0,n}}+\alpha\Delta)$ is verified by pulling this divisor back to $\M_{0,n}$, and then using the $F$-conjecture to check that it is nef and contracts only $\phi$-exceptional curves. To obtain an unconditional proof of Simpson's theorem, we will give a direct proof of the ampleness of the divisor $\phi_*(K_{\M_{0,n}}+\alpha\Delta)$ by showing that it lies in the interior of the nef cone of $\M_{0,\A}$.

In Lemma \ref{L:PushforwardClass}, we will see that
\[
\phi_*(K_{\M_{0,n}}+\alpha\Delta) \equiv \psi-2\Delta+\alpha(\Delta+\Delta_{s}),
\]
where $\psi, \Delta, \Delta_{s} \in N^{1}(\M_{0,\A})$ are certain tautological divisor classes on $\M_{0,\A}$. (See Section \ref{S:BasicProperties} for our notational conventions for divisor classes on $\M_{0,\A}$.) Using the idea sketched above and elaborated in Section \ref{S:Positivity}, we show that this linear combination of tautological divisor classes has positive intersection on any curve in $\M_{0,\A}$ (Theorem \ref{T:MainTheorem}).
\begin{mainresult}
If $\alpha \in \Q \cap (\frac{2}{k+2},\frac{2}{k+1}]$, the divisor class
$K_{\M_{0,\A}}+\phi_*\Delta$
has positive intersection on any 1-parameter family of $\A$-stable curves, $\A=\{1/k, \ldots, 1/k\}$.
\end{mainresult}
Minor variations on the proof of Theorem \ref{T:MainTheorem} show that $\psi-2\Delta+\alpha(\Delta+\Delta_{s})$ remains nef when we perturb by a small linear combination of boundary divisors of $\M_{0,\A}.$ Since the boundary divisors generate $\text{Pic}(\M_{0,\A})$, we conclude (Proposition \ref{P:Ampleness} and \ref{P:Ampleness2})
\begin{maincorollary}
If $\alpha \in \Q \cap (\frac{2}{k+2},\frac{2}{k+1}]$, the divisor class $K_{\M_{0,\A}}+\phi_*\Delta$ lies in the interior of the nef cone of $\M_{0,\A}$, and is therefore ample. In particular, Theorem 1.1 holds without assuming the $S_{n}$-equivariant $F$-conecture.
\end{maincorollary}

Let us give a brief outline of the contents of this paper. In Section \ref{S:BasicProperties}, we establish the notation that we will use for doing intersection theory on the spaces $\M_{0,\A}$. In particular, we define a number of tautological divisor classes on the weighted spaces $\M_{0,\A}$, and describe how they push-forward and pull-back under the natural reduction maps $\phi_{\A,\A'}:\M_{0,\A} \rightarrow \M_{0,\A'}$. In Section \ref{S:Positivity}, we focus on proving positivity statements for certain divisor classes on a 1-parameter family of $\A$-stable curves with smooth general fiber. It seems likely that the methods described here could be used to produce a number of new nef divisors on $\M_{0,n}$ or $\M_{0,\A}$, but in this paper we will simply focus on the divisor classes which are relevant for the log minimal model program. In Section 4, we prove {\em ampleness} of certain linear combinations of tautological divisors 
on spaces $\M_{0,\A}$, establishing Theorem \ref{T:MainTheorem} as a corollary.\\

\textbf{Acknowledgements.} The authors would like to thank Matthew Simpson for sharing a draft of his thesis, and for several informative conversations regarding its contents.

\section{Preliminaries on $\M_{0,\A}$}\label{S:BasicProperties}
In this section, we recall several facts about $\M_{0,\A}$ and establish the notation we will use for doing intersection theory on these spaces. We work over a fixed algebraically closed field $k$ (the characteristic of $k$ plays no role in our arguments whatsoever).

For any weight vector $\A=(a_1, \ldots, a_n) \in [0,1]^{n} \cap \Q$ satisfying $\sum_{i=1}^{n}a_i>2$, there exists a smooth projective variety $\M_{0,\A}$, which is a fine moduli space for the moduli problem of $\A$-stable curves of genus zero \cite{Hassett1}. Recall that a complete connected reduced nodal curve with $n$ smooth marked points $(C,p_1, \ldots, p_n)$ is $\A$-stable provided that:
\begin{itemize}
 \item[1.] If $p_{i_1}= \ldots= p_{i_k} \in C,$ then $\sum_{j=1}^{k}a_{i_j} \leq 1$.
\item[2.] $\omega_{C}(a_1p_1+\ldots+a_np_n)$ is ample.
\end{itemize}

The boundary $\M_{0,\A}\setminus M_{0,n}$ consists of divisors whose generic points parameterize curves where two sections collide, the union of these is denoted $\Delta_s$, and of divisors whose
generic point is a rational curve with two irreducible components and marked points $S_1\subset \{p_1,\dots, p_n\}$ on one component and marked points $S_2:=\{p_1,\dots,p_n\}\setminus S_1$ on the other component; these divisors are denoted $\Delta_{S_1,S_2}$.

For any pair of weight vectors $\A, \A'$ satsifying $a_{i} \leq a_i'$ for each $i=1, \ldots, n$, there is a birational reduction morphism
$$
\phi_{\A,\A'}:\M_{0,\A} \rightarrow \M_{0,\A'}.
$$
Note that the reduction morphism $\M_{0,n} \rightarrow \M_{0,\{1/2, \ldots, 1/2\}}$ is an isomorphism, while for $2 \leq k \leq \lfloor \frac{n-1}{2} \rfloor-1$, the morphism $\M_{0,\{1/k, \ldots, 1/k\}} \rightarrow \M_{0,\{1/(k+1), \ldots, 1/(k+1)\}}$ contracts all boundary divisors $\Delta_{S_1,S_2}$ satisfying $|S_1|\leq k+1$ or $|S_2| \leq k+1$.

We shall be concerned exclusively with weight vectors of the form
\begin{align*}
\A^{k}_{n,m}&:=\{\underbrace{1/k, \ldots, 1/k}_n, \underbrace{1, \ldots, 1}_m\},
\end{align*}
where $m \geq 0$, $k \geq 1$, and $m+n/k>2$. (These are the conditions under which one obtains a non-empty moduli problem.) Indeed, from now on, whenever we speak of a \emph{weight vector $\A$}, we mean that $\A=\A_{n,m}^{k}$ with $n,m,k$ satisfying these conditions. We will sometimes use the abbreviation $\A_{n}^{k}:=\A_{n,0}^{k}$.

The necessity of considering the $\A^{k}_{i,j}$-weight vectors stems from the inductive description of the boudary of $\M_{0,\A^{k}_n}$. For example, if $[n]=S_1 \cup S_2$ is a partition of $[n]$ into two subsets, with $|S_1|=n_1$ and $|S_2|=n_2$, then one has a natural isomorphism
$$\M_{0,\A^{k}_{n_1,1}} \times \M_{0,\A^{k}_{n_2,1}} \rightarrow \Delta_{S_1,S_2} \subset \M_{0,\A^{k}_n}.$$

Since $\M_{0,\A}$ is a smooth, there is a canonical isomorphism between $\text{Pic\,}(\M_{0,\A})$ and the group of Weil divisors modulo linear equivalence. Furthermore, linear equivalence and numerical equivalence on $\M_{0,\A}$ coincide, so we will typically consider all line-bundles and divisors on $\M_{0,\A}$ as divisor classes in $N^{1}(\M_{0,\A})$, defined up to numerical equivalence. Each reduction morphism $\phi_{\A,\A'}$ gives rise to well-defined push-forward and pull-back maps  on the space of divisors modulo numerical equivalence, induced by push-forward of cycles and pull-back of line bundles, respectively.

For $\A:=\A_{n,m}^{k}$, we let $(\pi:\C \rightarrow \M_{0,\A}, \{\sigma_i\}_{i=1}^{n}, \{\tau_j\}_{j=1}^{m})$ denote the universal curve, and define the following divisor classes (note that $\psi_{\tau}=0$ when $m=0$):
\begin{align*}
\psi_{\sigma}:&=-\sum_{i=1}^{n}\pi_*(\sigma_i^2),\\
\psi_{\tau}:&=-\sum_{i=1}^{m}\pi_*(\tau_i^2),\\
\Delta_{s}:&=\sum_{1 \leq i < j \leq n} \pi_*(\sigma_i.\sigma_j),\\
\Delta:&=\pi_*(\Sing(\pi)).\\
\end{align*}
Let us consider how these divisor classes push-forward under the natural reduction morphisms.
\begin{lemma}[Push-forward formulae]\label{L:PsiPush}
Let $\phi_{k}: \M_{0,n} \rightarrow \M_{0,\A_{n}^{k}}$ be the natural reduction morphism ($k \geq 2$). Then
\begin{align*}
(\phi_{k})_{*}\psi &= \psi+2\Delta_s \in N^{1}\bigl(\M_{0,\A_{n}^{k}}\bigr),\\
(\phi_{k})_{*}\Delta&= \Delta+\Delta_s \in N^{1}\bigl(\M_{0,\A_{n}^{k}}\bigr).
\end{align*}
\end{lemma}
\begin{proof}
We will prove that $\phi_{2}:\M_{0,n} \rightarrow \M_{0,\A_{n}^{2}}$ satisfies
\begin{align*}
(\phi_2)_{*}\psi &= \psi+2\Delta_s,\\
(\phi_2)_{*}\Delta&= \Delta+\Delta_s,\\
\intertext{
and that $\phi_{2,k}:\M_{0,\A_{n}^{2}} \rightarrow \M_{0,\A_{n}^{k}}$ satisfies
}
(\phi_{2,k})_{*}\psi &= \psi,\\
(\phi_{2,k})_{*}\Delta&= \Delta.
\end{align*}
The latter formulae are immediate from the fact that the locus in $\M_{0,\A_{n}^{k}}$ over which the universal curve fails to be $\A_{n}^{2}$-stable has codimension $\geq 2$ (it is precisely the locus where three sections collide). The same reasoning shows that
\begin{align*}
(\phi_2)_{*}\psi &= \psi+a\Delta_s,\\
(\phi_2)_{*}\Delta&= \Delta+b\Delta_s,
\end{align*}
for some $a, b \in \Q$. Indeed, it is equivalent to show that $\phi_*\psi = \psi, \phi_*\Delta = \Delta$ as cycles in $\SM_{0,\A_n^{2}} \backslash \Delta_{s}$. But this is immediate since the locus of non-stable curves in $\SM_{0,\A_{n}^{2}} \backslash \Delta_{s}$ has codimension $\geq 2$. To establish that $a=2$ and $b=1$, we will use the following test curve.

Let $(\C \rightarrow B, \{\sigma_i\}_{i=1}^{n})$ be the complete 1-parameter family of $\A_{n}^2$-stable curves, obtained by taking $\C \rightarrow B$ to be the projection $\P^{1} \times \P^{1} \rightarrow \P^{1}$, taking  $\sigma_{1}, \ldots, \sigma_{n-1}$ to be $n-1$ distinct constant sections, and taking $\sigma_n$ to be the diagonal section. We have the following intersection numbers on $\M_{0,\A_n^2}$.
\begin{align*}
\psi.B&=2,\\
\Delta.B&=0,\\
\Delta_{s}.B&=n-1.
\end{align*}

Now let $(\C^{s} \rightarrow B^{s}, \{\sigma_i^{s}\}_{i=1}^{n})$ be the stable curve over the same base obtained by blowing-up the intersection points $\{\sigma_i \cap \sigma_n\}_{i=1}^{n-1}$ and taking $\sigma_i^{s}$ to be the strict transforms of $\sigma_i$. On $\M_{0,n}$, we have
\begin{align*}
\psi.B^{s}&=4+2n,\\
\Delta.B^{s}&=n-1,\\
\Delta_{s}.B^{s}&=0
\end{align*}

Since $\phi_{2}:\M_{0,n} \rightarrow \M_{0,\A_{n}^2}$ is an isomorphism, mapping $B^{s}$ isomorphically onto $B$, we must have $\psi.B^{s}=((\phi_2)_*\psi).B$ and $\Delta.B^{s}=((\phi_2)_*\Delta).B$. Thus,
\begin{align*}
4-2n&=(\psi + a\Delta_{s}).B=2+a(n-1),\\
n-1&=(\Delta+b\Delta_{s}).B=b(n-1),
\end{align*}
from which we conclude that $a=2$ and $b=1$ as desired.
\end{proof}

Now we can express the divisors $\phi_{*}(K_{\M_{0,n}}+\alpha\Delta)$ in terms of the fundamental classes defined above. In our subsequent analysis, it will be convenient to rescale the resulting $\Q$-divisor in order to ensure that the coefficient of $\Delta$ is $-1$.
\begin{lemma}\label{L:PushforwardClass} For any $\alpha \in \Q \cap [0,1]$, set $c=1/(2-\alpha)$. Let $\phi: \M_{0,n} \rightarrow \M_{0,\A_{n}^{k}}$ be the natural morphism. Then we have
\begin{itemize}
\item[(1)]$K_{\M_{0,n}}+\alpha\Delta=\psi-2\Delta+\alpha\Delta \sim c\psi-\Delta \in \N^{1}(\M_{0,n})$,
\item[(2)] $\phi_*(K_{\M_{0,n}}+\alpha\Delta)=\psi-2\Delta+\alpha\Delta \sim c\psi+(2c-1)\Delta_s-\Delta \in \N^{1}(\M_{0,\A_{n}^{k}})$,
\end{itemize}
where `$\sim$' denotes numerical proportionality.
\end{lemma}
\begin{proof}
Using the Grothendieck-Riemann-Roch formula as in \cite{Harris}, the canonical class of the stack $\SM_{g,n}$ is $13\lambda-2\Delta+\psi$. Since $\SM_{0,n}=\M_{0,n}$ and $\lambda=0 \in \text{Pic}(\M_{0,n})$, we have $K_{\M_{0,n}}=\psi-2\Delta$. The lemma is now an immediate consequence of the push-forward formulae
\begin{align*}
\phi_{*}\psi &\equiv \psi+2\Delta_s,\\
\phi_{*}\Delta&\equiv \Delta+\Delta_s,
\end{align*}
which are proved in Lemma \ref{L:PsiPush}.
\end{proof}
Since divisors of the form $c\psi+(2c-1)\Delta_s-\Delta$ play a distinguished role in our subsequent analysis, it will be useful to make the following definition.
\begin{definition}[]\label{L:FundamentalDivisor}
For any weight vector $\A:=\A_{n,m}^{k}$ and rational number $c \in \Q$, we define the divisor class $D_{k}(c)$ on $\M_{0,\A}$ by the formula
$$
D_{k}(c):=c\psi_\sigma+(2c-1)\Delta_s+\psi_\tau-\Delta.
$$
\end{definition}

Next, we study how the tautological divisor classes pull-back under the reduction morphisms. 
\begin{lemma}[Pull-back formulae for $\phi$]\label{L:Pullback}
Consider the reduction morphism $$\phi: \M_{0,\A_{n,m}^{k-1}}\ra \M_{0,\A_{n,m}^{k}},$$ and let $F$ denote the union of the exceptional divisors of $\phi$. Then we have
\begin{enumerate}
\item $\phi^{*} \psi_\sigma=\psi_\sigma-k F$,
\item $\phi^{*} \psi_\tau = \psi_\tau$,
\item $\phi^{*} \Delta_s=\Delta_s+\binom{k}{2}F$,
\item $\phi^{*} \Delta=\Delta-F$.
\end{enumerate}
\end{lemma}
\begin{proof} 
Clearly, 
$$\phi^{*} \psi_\sigma =\psi_\sigma+aF,$$ 
for some constant $a$. To determine this coefficient, we use a test curve contracted by $\phi$. Let $\C_1$ be the blow-up of $\P^2$ at a point with exceptional curve $E$, and regard $E$ as a section of the natural $\P^{1}$-bundle fibration $\C_{1} \rightarrow B=\P^{1}.$ Let $k$ sections of $\C_1\ra B$ be given by the strict transforms 
of $k$ lines on $\C_1$; note that each section has self-intersection $1$.  Let $\C_2=\P^1\times \P^1$ 
and take $n+m-k$ constant sections of the second projection $\C_2\ra B=\P^{1}$. Let $\C$ be the union of $\C_1$ and $\C_2$, obtained
by identifying $E \subset \C_1$ with a constant section of $\C_2\ra B$. 
Then we can regard $\C\ra B$ as the family of $\A_{n,m}^{k-1}$\nb-stable curves 
by giving all $k$ sections of $\C_1$ and $n-k$ sections of $\C_2$ weight $1/(k-1)$, 
and giving the remaining $m$ sections of $\C_2$
weight $1$. 

Clearly, $B\subset \M_{0,\A^{k-1}_{n,m}}$ is contracted by $\phi$. 
We calculate 
\begin{align*}
\psi_\sigma .B &=-k, \\
\psi_\tau.B&=0, \\
 \Delta_s .B&= \binom{k}{2}, \\
\Delta.B &=-1,\\
F.B &=-1.
\end{align*}
By the projection formula
$$0 =B\cdot \phi^{*} \psi_\sigma=\psi_\sigma. B+a(F. B)=-k-a.$$
Therefore, $a=-k$. The remaining formulae are proved in the same fashion.
\end{proof}

Finally, we will also need to utilize the natural morphism
$$
\chi:\M_{0,\A_{n-k,m+1}^{k}} \rightarrow \M_{0,\A_{n,m}^{k}},
$$
obtained by replacing the $(m+1)^{st}$ section of weight one with $k$ coincident sections of weight $1/k$. Let $(\pi:\C\ra \M_{0,\A_{n-k,m+1}^{k}}, \{\sigma_i\}_{i=1}^{n-k}, \{\tau_j\}_{j=1}^{m+1})$ be the universal 
curve, and consider the divisor class $\psi_{\tau_{m+1}}:=-\pi_*(\tau_{m+1}^2)$ on $\M_{0,\A_{n-k,m+1}^{k}}$. 
\begin{lemma}[Pull-back formulae for $\chi$]\label{L:PullbackX}
Under the replacement morphism $\chi:\M_{0,\A_{n-k,m+1}^{k}} \rightarrow \M_{0,\A_{n,m}^{k}}$, the tautological divisors pull-back as
\begin{enumerate}
\item $\chi^{*} \psi_\sigma=\psi_\sigma+k\psi_{\tau_{m+1}}$,
\item $\chi^{*} \psi_\tau = \psi_\tau-\psi_{\tau_{m+1}}$,
\item $\chi^{*} \Delta_s=\Delta_s-\binom{k}{2}\psi_{\tau_{m+1}}$,
\item $\chi^{*} \Delta=\Delta$.
\end{enumerate}
\end{lemma}
\begin{proof}
 Observe that $\pi:\C\ra \M_{0,\A_{n-k,m+1}^{k}}$ together with sections 
 $\{\sigma'_i\}_{i=1}^{n}$ and $\{\tau_j\}_{j=1}^{m}$, where $\sigma_i'=\sigma_i$ for $1\leq i\leq n-k$ and $\sigma'_i=\tau_{m+1}$ for $n-k+1\leq i\leq n$, is a family of $\A_{n,m}^{k}$\nb-stable curves.
 The morphism $\chi$ is induced by this family and is a closed embedding. 
 The lemma now follows from the projection formula and from the following list of equalities
 \begin{align*}
 \chi^{*} \psi_\sigma &=-\sum_{i=1}^{n} \pi_*(\sigma'_i)^2 =-\sum_{i=1}^{n-k} \pi_* (\sigma_i^2) - k \pi_* (\tau_{m+1}^2) = \psi_\sigma+k\psi_{\tau_{m+1}},\\
 \chi^{*} \psi_\tau &=-\sum_{j=1}^{m} \pi_*(\tau_j^2) =-\sum_{j=1}^{m+1} \pi_*(\tau_j^2)+\pi_*(\tau_{m+1}^2)= \psi_\tau-\psi_{\tau_{m+1}}, \\
 \chi^{*} \Delta_s &=\sum_{1\leq i<j\leq n}\pi_*(\sigma'_i\cdot \sigma'_j)=\sum_{1\leq i<j\leq n-k}\pi_*(\sigma_i\cdot \sigma_j)+\binom{k}{2}\pi_*(\tau_{m+1}^2) =\Delta_s-\binom{k}{2}\psi_{\tau_{m+1}}, \\
 \chi^{*} \Delta&=\pi_*(\Sing(\pi))=\Delta.
 \end{align*}
In the third equality, we use the fact that $\tau_{m+1}$ does not intersect any of the sections $\{\sigma_i\}_{i=1}^{n-k}$.

\end{proof}

\section{Positivity on 1-parameter families}\label{S:Positivity}

Throughout this section, we suppose that $(\pi: \C \rightarrow B, \sigman, \taum)$ is a {\em generically smooth} family of $\A:=\A_{n,m}^{k}$-stable pointed curves over a smooth curve $B$.
 In particular, up to $k$ of the sections $\sigman$ can collide, while the sections $\taum$ are each disjoint from each other and the $\sigman$. Note that we  allow the case $m=0$, i.e. $\A:=\A_{n}^{k}$. Under these conditions, we will explain how to derive inequalities among the intersection numbers:
\begin{align*}
\psi_{\sigma}.B&:=-\sum_{i=1}^{n}\sigma_i^{2} \, ,\\
\psi_{\tau}.B&:=-\sum_{i=1}^{m}\tau_i^{2}\, ,\\
\Delta_s.B&:=\sum_{1 \leq j < k \leq n}\sigma_i.\sigma_j\, ,\\
\Delta.B&:=\deg(\pi_*(\Sing(\pi)))\, .
\end{align*}

To start, we let $\C_{0}$ be a minimal resolution of singularities of $\C$. By successively blowing-down (-1)-curves contained in the fibers, we obtain a sequence of birational morphisms over $B$
$$\C_{0} \rightarrow \C_{1} \rightarrow \ldots \rightarrow \C_{N},$$
such that $\C_{N}$ is a smooth $\P^{1}$-bundle over $B$. If we let $\{\sigma^{i}_j\}_{j=1}^{n}$ and $\{\tau^{i}_j\}_{j=1}^{m}$ denote the sections of $\pi_i: \C_{i} \rightarrow B$ obtained as the images of $\sigman$ and $\taum$ on $\C_{i}$, then each family $(\pi_i: \C_{i} \rightarrow B, \{\sigma^{i}_j\}_{j=1}^{n}, \{\tau^{i}_j\}_{j=1}^{m})$ satisfies the following conditions:
\begin{itemize} 
\item[1.] The geometric fibers of $\pi$ are reduced connected nodal curves of arithmetic genus zero.
\item[2.] The generic fiber of $\pi$ is smooth.
\item[3.] The sections $\sigman$ and $\taum$ lie in the smooth locus of $\pi$.
\item[4.] If the exceptional divisor of $\C_{i} \rightarrow \C_{i+1}$ meets $r_1$ sections of weight $1/k$ and $r_2$ sections of weight 1, then $r_1/k+r_2 >1.$
\end{itemize}

With notation as above, we define functions $F_{\Delta}, F_{\sigma}, F_{\tau}, F_{\sigma,\tau}: [0,N] \rightarrow \Q$ by the formulae
\begin{align*}
F_{\Delta}(i)&:=\deg (\Sing(\pi_i))\, ,\\
F_{\sigma}(i)&:=-\frac{1}{n-1}\sum_{1 \leq j < k \leq n}(\sigma^i_j-\sigma^i_k)^2\, ,\\
F_{\tau}(i)&:=-\frac{1}{m-1}\sum_{1 \leq j < k \leq m}(\tau^i_j-\tau^i_k)^2\, ,\\
F_{\sigma, \tau}(i)&:=-\frac{1}{nm}\sum_{j=1}^{n}\sum_{k=1}^{m}(\sigma^i_j-\tau^i_k)^2 \, .\\
\end{align*}

The following lemma is the key ingredient in all our subsequent arguments.
\begin{lemma}[Key Lemma]\label{L:Main}
\begin{itemize}
\item[]
\item[(a)] The numerical functions $F_{\Delta}(i), F_{\sigma}(i), F_{\tau}(i), F_{\sigma, \tau}(i)$ satisfy
\begin{align*}
F_{\Delta}(0)&=\Delta.B,&F_{\Delta}(N)=0,\\
F_{\sigma}(0)&=\psi_{\sigma}.B+\frac{2}{n-1}\Delta_s.B,&F_{\sigma}(N)=0,\\
F_{\tau}(0)&=\psi_{\tau}.B,&F_{\tau}(N)=0,\\
F_{\sigma, \tau}(0)&=\frac{1}{n}\psi_{\sigma}.B+\frac{1}{m}\psi_{m}.B,&F_{\sigma, \tau}(N)=0.\\
\end{align*}
\item[(b)] Suppose that the exceptional divisor of the birational map $\C_{i} \rightarrow \C_{i+1}$ meets $r_1$ of the sections $\sigman$ and $r_2$ of the sections $\taum$. Then we have
\begin{align*}
F_{\Delta}(i)-F_{\Delta}(i+1)&=1,\\
F_{\sigma}(i)-F_{\sigma}(i+1)&= \frac{r_1(n-r_1)}{n-1}, \\
F_{\tau}(i)-F_{\tau}(i+1)&=\frac{r_2(m-r_2)}{m-1},\\
F_{\sigma, \tau}(i)-F_{\sigma, \tau}(i+1)&=\frac{r_1(m-r_2)+r_2(n-r_1)}{nm}. 
\end{align*}
\end{itemize}

\end{lemma}
\begin{proof}
We first prove part (a). To see that
$$
F_{\sigma}(0)=F_{\tau}(0)=F_{\sigma,\tau}(0)=0,
$$
simply observe that on a $\P^{1}$-bundle, the difference of two sections is numerically equivalent to a multiple of the fiber class, and the fiber class has self-intersection zero. Also, $F_{\Delta}(0)=0$ is clear, since a $\P^{1}$-bundle has no singular fibers.

To see that $F_{\Delta}(0)=\Delta.B$, it is sufficient to observe that the minimal desingularization $\C^{0} \rightarrow \C$ has the effect of replacing nodes where the total space has an $A_{k}$-singularity (thus contributing $k$ to $\deg(\Sing(\pi))$) by $k$ nodes with smooth total space, each of which contributes one to $\deg(\Sing(\pi_0))$.

To see that $F_{\sigma}(0)=\psi_{\sigma}+\frac{2}{n-1}\Delta_s$, note that since the sections $\sigman, \taum$ are disjoint from the singular locus of $\pi$, their self-intersections are not changed by taking a minimal desingularization. Thus,
\begin{align*}
F_{\sigma}(0)&=-\frac{1}{n-1}\sum_{1 \leq i<j \leq n}(\sigma^0_j-\sigma^0_k)^2=-\frac{1}{n-1}\sum_{1 \leq i<j \leq n}(\sigma_j-\sigma_k)^2\\
&=-\sum_{i=1}^{n}\sigma_i^2 +\frac{2}{n-1}\sum_{1 \leq i<j \leq n} \sigma_i.\sigma_j=(\psi_{\sigma}.B)+\frac{2}{n-1}(\Delta_s.B).
\end{align*}
The computations for $F_{\sigma}(0)$ and $F_{\sigma, \tau}(0)$ are similar, bearing in mind the fact that all intersections $\tau_i.\sigma_j$ and $\tau_i.\tau_j$ ($i \neq j$) are zero. We have
\begin{align*}
F_{\tau}(0)&=-\frac{1}{n-1}\sum_{1 \leq i<j \leq n}(\tau^0_j-\tau^0_k)^2=-\frac{1}{n-1}\sum_{1 \leq i<j \leq n}(\tau_j-\tau_k)^2\\
&=-\sum_{i=1}^{n}\tau_i^2 =(\psi_{\tau}.B)\\
F_{\sigma, \tau}(0)&=-\frac{1}{nm}\sum_{j=1}^{n}\sum_{k=1}^{m}(\sigma^0_j-\tau^0_k)^2=-\frac{1}{nm}\sum_{j=1}^{n}\sum_{k=1}^{m}(\sigma_j-\tau_k)^2\\
&=-\frac{1}{nm}\sum_{i=1}^{n}m\sigma_i^2 -\frac{1}{nm}\sum_{i=1}^{m}n\tau_i^2=\frac{1}{n}(\psi_{\sigma}.B)+\frac{1}{m}(\psi_{\tau}.B)\\
\end{align*}

It remains to prove part (b) of the lemma. Let $\phi:\C_{i} \rightarrow \C_{i+1}$ denote the $i^{th}$ blow-down morphism. Since $\phi$ contracts a single (-1)-curve, it is clear that $\C_{i+1}$ has one less node in the union of its singular fibers than $\C_{i}$, i.e. $F_{\Delta}(i)-F_{\Delta}(i+1)=1$.

We may assume without loss of generality that the exceptional divisor $E$ of $\phi$
meets $\{\sigma_j\}_{j=1}^{r_1}$ and $\{\tau_j\}_{j=1}^{r_2}$, and is disjoint from $\{\sigma_j\}_{j=r_1+1}^{n}$ and $\{\tau_j\}_{j=r_2+1}^{m}$. Then we have
\begin{align*}
\sigma_{j}^{i}-\sigma_{k}^{i}&=\phi^*(\sigma_{j}^{i+1}-\sigma_{k}^{i+1}), &&\text{ if }
j,k \leq r_1,\text{ or } j,k \geq r_1+1,\\
\sigma_{j}^{i}-\sigma_{k}^{i}&=\phi^*(\sigma_{j}^{i+1}-\sigma_{k}^{i+1})+E, &&\text{ otherwise.}
\end{align*}
It follows that
\begin{align*}
(\sigma_{j}^{i}-\sigma_{k}^{i})^2&=(\sigma_{j}^{i+1}-\sigma_{k}^{i+1})^2, &&\text{ if }
j,k \leq r_1,\text{ or } j,k \geq r_1+1,\\
(\sigma_{j}^{i}-\sigma_{k}^{i})^2&=(\sigma_{j}^{i+1}-\sigma_{k}^{i+1})^2-1, &&\text{ otherwise.}
\end{align*}
Thus,
\begin{align*}
F_{\sigma}(i)-F_{\sigma}(i+1)&=\frac{1}{n-1}\sum_{1 \leq j<k \leq n}\left((\sigma^{i+1}_j-\sigma^{i+1}_k)^2-(\sigma^i_j-\sigma^i_k)^2\right)=\frac{r_1(n-r_1)}{n-1}.
\end{align*}
The remaining statements in part (b) are proved in the same fashion.
\end{proof}

How can we use this lemma to prove positivity of certain linear combinations $\alpha \psi_{\sigma}.B + \beta \psi_{\tau}.B + \gamma \Delta_s.B + \eta \Delta.B$? Well, part (a) implies that we can find coefficients $a_{\sigma}, a_{\tau}, a_{\sigma, \tau}, a_{\delta} \in \Q$, such that
\[
a_{\sigma}F_{\sigma}(0)+a_{\tau}F_{\tau}(0)+a_{\sigma, \tau}F_{\sigma, \tau}(0)+a_{\delta}F_{\Delta}(0)=\alpha (\psi_{\sigma}.B) + \beta (\psi_{\tau}.B) + \gamma (\Delta_s.B) + \eta (\Delta.B)
\]
On the other hand, we have
\[
a_{\sigma}F_{\sigma}(N)+a_{\tau}F_{\tau}(N)+a_{\sigma, \tau}F_{\sigma, \tau}(N)+a_{\delta}F_{\Delta}(N)=0.
\]
Thus, if we can show that $a_{\sigma}F_{\sigma}(i)+a_{\tau}F_{\tau}(i)+a_{\sigma, \tau}F_{\sigma, \tau}(i)+a_{\delta}F_{\Delta}(i)$ is a decreasing function of $i$, we may conclude that
$$\alpha (\psi_{\sigma}.B) + \beta (\psi_{\tau}.B) + \gamma (\Delta_s.B) + \eta (\Delta.B) \geq 0,$$
with equality holding only if $N=0$.

Thus, the key point is to determine what conditions on the coefficients $a_{\sigma}, a_{\tau}, a_{\sigma, \tau}, a_{\delta}$  make the function $a_{\sigma}F_{\sigma}(i)+a_{\tau}F_{\tau}(i)+a_{\sigma, \tau}F_{\sigma, \tau}(i)+a_{\delta}F_{\Delta}(i)$ decreasing. There are three free parameters up to scaling, and part (b) of the lemma allows one to describe the polytope where this function is decreasing. In practice, we are mainly interested in the subspace of divisors of the form $\alpha (\psi_{\sigma}.B) + \beta (\psi_{\tau}.B) + \gamma (\Delta_s.B) - \beta (\Delta.B)$ since $\psi_{\tau}-\Delta$ is is functorial with respect to the boundary stratification (see Lemma \ref{L:BoundaryInduction}), whereas $\psi_{\tau}$ and $\Delta$ individually are not. Thus, we are left with the problem of describing a two-dimensional polytope, and this is what we do in the following proposition.

In order to make the numerics tractable, we divide the analysis into several cases. Note that the one case where we get a sharp inequality is case (2.), where we have one section of weight $1$ and $k+1$ sections of weight $1/k$ on a $\P^{1}$-bundle. This corresponds to the fact that the moving components of the extremal rays contracted by the map $\M_{0,\A_{n}^{k}} \rightarrow \M_{0,\A_{n}^{k+1}}$ take precisely this form, and eventually we will see that this inequality determines the precise value of $\alpha$ at which $\M_{0,n}(\alpha)$ transforms from $\M_{0,\A_{n}^{k}}$ into $\M_{0,\A_{n}^{k+1}}$.
\begin{proposition}\label{P:Positivity}
Let $(\C \rightarrow B, \sigman, \taum)$ be an arbitrary complete 1-parameter family of $\A$-stable curves with smooth general fiber. Then the intersection numbers $\psi_{\sigma}.B, \psi_{\tau}.B, \Delta_{s}.B, \Delta.B$ satisfy the following inequalities.
\begin{itemize}
\item[1.]If $m=0$, then for $a > \frac{n-1}{(n-k-1)(k+1)}$ we have
\begin{align*}
a(\psi_{\sigma}.B)+\frac{2a}{n-1}(\Delta_{s}.B)-(\Delta.B) > 0.
\end{align*}
\item[2.] If $m=1$, then for $a > \frac{n-1}{n(k+1)}$ we have
\begin{align*}
\left(a+\frac{1}{n}\right)(\psi_{\sigma}.B)+\frac{2a}{n-1}(\Delta_{s}.B)+(\psi_{\tau}.B)-(\Delta.B) \geq 0.
\end{align*}
Furthermore, equality holds if and only if $n=k+1$.\\
\item[3.]If $m\geq 2,\, 2 \leq n \leq k$, then for $a > 0$ and $b >0$ we have
\begin{align*}
\left(a+\frac{b}{n}\right)(\psi_{\sigma}.B)+\frac{2a}{n-1}(\Delta_{s}.B)+(\psi_{\tau}.B)-(\Delta.B) > 0.
\end{align*}
\item[4.] If $m\geq 2,\, n \geq k+1$, then for $\frac{(k+1)(n-k-1)}{n-1}a+\frac{k+1}{n}b>1\text{ and }b>1$ we have 
\begin{align*}
\left(a+\frac{b}{n}\right)(\psi_{\sigma}.B)+\frac{2a}{n-1}(\Delta_{s}.B)+(\psi_{\tau}.B)-(\Delta.B) > 0.
\end{align*}
\end{itemize}
\end{proposition}
\begin{proof} 

\hfill

(1.) We have
$$
aF_{\sigma}(0)-F_{\Delta}(0)=a(\psi_{\sigma}.B)+\frac{2a}{n-1}(\Delta_{s}.B)-(\Delta.B).
$$ 
Thus, it suffices to show that $G(i):=aF_{\sigma}(i)-F_{\Delta}(i)$ is a decreasing function of $i$. The exceptional divisor of $\C_{i} \rightarrow \C_{i+1}$ meets at least $k+1$ sections, so Lemma \ref{L:Main} (b) implies
\[
aF_{\sigma}(i)-F_{\Delta}(i)-\left(aF_{\sigma}(i+1)-F_{\Delta}(i+1) \right)=\frac{r_1(n-r_1)}{n-1}a-1\geq \frac{(k+1)(n-k-1)}{n-1}a-1>0,
\]
when $a > \frac{n-1}{(n-k-1)(k+1)}$.
\hfill

(2.) We have 
 $$
aF_{\sigma}(0)+F_{\sigma,\tau}(0)-F_{\Delta}(0)=(a+\frac{1}{n})(\psi_{\sigma}.B)+\frac{2a}{n-1}(\Delta_{s}.B)+(\psi_{\tau}.B)-(\Delta.B).
$$ 
Thus, it suffices to show that $G(i):=aF_{\sigma}(i)+F_{\sigma,\tau}(i)-F_{\Delta}(i)$ is a descreasing function of $i$. If $n=k+1$ and $m=1$, then the hypothesis of $\A$-stability implies that the original family $\C \rightarrow B$ has no singular fibers, i.e. $\C=\C_{0}$ is already a $\P^{1}$-bundle. In this case,
$$
G(0)=G(N)=0.
$$
Thus, we may assume that $n\geq k+2$, so the exceptional divisor of $\C_{i+1} \rightarrow \C_{i}$ meets $n-1\geq r_1 \geq  k+1$ sections of weight $1/k$. In addition, since there is only one section of weight 1, and we may always choose to blow-down a (-1)-curve disjoint from this section, we may assume that the exceptional divisor  $\C_{i+1} \rightarrow \C_{i}$ meets no sections of weight 1. Thus,
\begin{align*}
G(i)-G(i+1) &=\frac{r_1(n-r_1)}{n-1}a+\frac{r_1}{n}-1\\ 
&\geq \min\left\{\frac{(k+1)(n-k-1)}{n-1}a+\frac{k+1}{n}-1,
a+\frac{n-1}{n}-1\right\}>0,
\end{align*}
when $a > \frac{n-1}{n(k+1)}$.\\

(3.) We have
\[
aF_{\sigma}(0)+bF_{\sigma,\tau}(0)+\frac{m-b}{m}F_{\tau}(0)-F_{\Delta}(0)=(a+\frac{b}{n})(\psi_{\sigma}.B)+\frac{2a}{n-1}(\Delta_{s}.B)+(\psi_{\tau}.B)-(\Delta.B). 
\]
Thus, it suffices to show that the function $G(i):=aF_{\sigma}(i)+bF_{\sigma,\tau}+\frac{m-b}{m}F_{\tau}(i)-F_{\Delta}(i)$ is a decreasing function of $i$. Suppose that the exceptional
divisor of $\C_{i}\ra \C_{i+1}$ meets $r_1$ sections of weight $k$ and $r_2$ sections of weight $1$. We have
\[
G(i)-G(i+1)=\frac{r_1(n-r_1)}{n-1}a + \frac{r_1(m-r_2)+r_2(n-r_1)}{mn}b+\frac{r_2(m-r_2)(m-b)}{m(m-1)}-1
\]
Denote the right-hand side by $\H(r_1,r_2)$. Then $\H$ is a convex function in each variable $r_1$ and $r_2$ (but not 
necessarily in both) and is symmetric about the point
$(n/2,m/2)$. The integer pairs $(r_1,r_2)$ satisfy the inequalities
\begin{itemize}
\item[(i)] $1 \leq r_1 \leq n$,
\item[(ii)] $1\leq r_2 \leq m$,
\item[(iii)]  $r_1/k+r_2>1,$
\item[(iv)] $(n-r_1)/k+(m-r_2)>1.$
\end{itemize}  
When $m\geq 3$, the integers $(r_1,r_2)$ that satisfying (i)-(iv) are easily seen to lie in the convex hull of points $(0,2), (1,1), (n,1)$ and points
$(n,m-2), (n-1,m-1), (0,m-1)$ (the second triple is the reflection of the first triple about $(n/2,m/2)$). 
For $m=2$, the integers $(r_1,r_2)$ are in the convex hull of $(1,1)$ and $(n-1,1)$.

By convexity of $\H$ in each argument we have $$\H(r_1,r_2)\geq \min\{ \H(0,2), \H(1,1), \H(n,1)\}$$ for $m\geq 3$,
and 
$$\H(r_1,r_2)\geq \H(1,1)$$ for $m=2$. To prove the statement of the proposition, we calculate
\begin{align*}
\H(1,1) &=a+\frac{m-2}{mn}b>0, \\
\H(0,2) &=\frac{2b}{m(m-1)}+\frac{m-3}{m-1}>0 \ \ \text{when $m\geq 3$}, \\
\H(n,1) &=\frac{m-2}{m}b>0 \ \ \text{when $m\geq 3$}.
\end{align*}

\hfill

(4.) As in Part (3.), we have
\[
aF_{\sigma}(0)+bF_{\sigma,\tau}(0)+\frac{m-b}{m}F_{\tau}(0)-F_{\Delta}(0)=(a+\frac{b}{n})(\psi_{\sigma}.B)+\frac{2a}{n-1}(\Delta_{s}.B)+(\psi_{\tau}.B)-(\Delta.B). 
\]
To show that the function $G(i):=aF_{\sigma}(i)+bF_{\sigma,\tau}+\frac{m-b}{m}F_{\tau}(i)-F_{\Delta}(i)$ is a decreasing function of $i$, we recall that
if the exceptional
divisor of $\C_{i}\ra \C_{i+1}$ meets $r_1$ sections of weight $k$ and $r_2$ sections of weight $1$, then
\[
G(i)-G(i+1)=\H(r_1,r_2),
\]
where $\H(r_1,r_2)$ is as defined in the proof of part (3.). The integer pairs $(r_1,r_2)$ satisfy the same inequalities (i)-(iv) as in part (3.). By the convexity of $\H$ in each factor we have 
\begin{align*}
\H(r_1,r_2) &\geq \min\bigl\{ \H(0,m), \H(0,2), \H(1,1), \H(k+1,0),  \H(n,0), \H(n,m-2),  \\ & \H(n-1,m-1), \H(n-k-1,m) \bigr\}.
\end{align*}
The statement of the proposition now follows from the following computations
\begin{align*}
\H(0,m) &=\H(n,0)=b-1 \\
\H(0,2) &=\H(n,m-2)=\frac{2b}{m(m-1)}+\frac{m-3}{m-1}\\
\H(1,1) &=\H(n-1,m-1)=a+\frac{m-2}{mn}b\\ 
\H(k+1,0) &=\H(n-k-1,m)=\frac{(k+1)(n-k-1)}{n-1}a+\frac{k+1}{n}b-1.
\end{align*}
\end{proof}

As we indicated in the introduction, we are really interested in divisors of the form
$D(c):=c\psi+(2c-1)\Delta_{s}+\psi_{\tau}-\Delta$. In the following corollary, we will simply record the values of $c$ such that the ray
$$
\{c\psi+(2c-1)\Delta_{s}+\psi_{\tau}-\Delta: c \in \Q \}
$$
intersects the polytope described by the previous proposition.

\begin{corollary}\label{C:Positivity} Let $\A:=\A_{n,m}^{k}$ be a weight vector with $k\geq 2$. If $(\C \rightarrow B, \sigman, \taum)$ is any complete 1-parameter family of $\A$-stable curves with smooth general fiber, we have
$$D_k(c).B:=c(\psi_{\sigma}.B)+(2c-1)\Delta_{s}.B+\psi_{\tau}.B-\Delta.B > 0$$
 in the following cases:
  \begin{align*}
1.\,\, &c=\frac{n-1}{2(n-2)} && \text{when} && m=0,\\
2.\,\, &c=\frac{n+1}{2n} && \text{when} && m=1, n\geq k+2,\\
3.\,\, &1/2<c\leq \frac{k+2}{2(k+1)} && \text{when} && m\geq 2,\, 2 \leq n \leq k,\\
4.\,\, &1/2<c<\frac{n+1}{2n} \text{ with $1<b<n/2$},&& \text{when} && m\geq2,\, n \geq k+1.\\
\intertext{In addition, we have
$
D_k(c).B= 0
$
in the case}
5.\,\, & c=\frac{k+2}{2(k+1)}&& \text{when} && m=1, n=k+1.
\end{align*}
\end{corollary}

\begin{proof}

\hfill

(1.) Observe that if $a=c=\frac{n-1}{2(n-2)}$, 
\begin{align*}
D_k(c)=a(\psi_{\sigma}.B)+\frac{2a}{n-1}(\Delta_{s}.B)-(\Delta.B).
\end{align*}
Note that $a=\frac{n-1}{2(n-2)}> \frac{n-1}{(n-k-1)(k+1)}$ for $k\geq 2$, so the statement follows from Proposition \ref{P:Positivity} (1.).\\

(2.) Observe that if $a=c-1/n=\frac{n-1}{2n}$, 
\begin{align*}
D_k(c)=(a+\frac{1}{n})(\psi_{\sigma}.B)+\frac{2a}{n-1}(\Delta_{s}.B)+(\psi_{\tau}.B)-(\Delta.B).
\end{align*}
Note that $a=\frac{n-1}{2n}\geq \frac{n-1}{n(k+1)}$, and the inequality is strict for $k\geq 2$, so the statement follows 
from Proposition \ref{P:Positivity} (2.).\\

\item(3.) Take $a=(n-1)(c-1/2)$ and $b=n(\frac{n-1}{2}-(n-2)c)$. Then 
\begin{align*}
D_k(c)=(a+\frac{b}{n})(\psi_{\sigma}.B)+\frac{2a}{n-1}(\Delta_{s}.B)+(\psi_{\tau}.B)-(\Delta.B).
\end{align*}
Since $k\geq 2$ and $n\leq k$, for $1/2<c\leq \frac{k+2}{2(k+1)}$ we have $a,b>0$, so the statement follows 
from Proposition \ref{P:Positivity} (3.).\\
\item(4.)
Take $a=(n-1)(c-1/2)$ and  $b=n(\frac{n-1}{2}-(n-2)c)$. Then 
\begin{align*}
D_k(c)=(a+\frac{b}{n})(\psi_{\sigma}.B)+\frac{2a}{n-1}(\Delta_{s}.B)+(\psi_{\tau}.B)-(\Delta.B).
\end{align*}
Note that $a$ and $b$ satisfy the assumptions of Proposition \ref{P:Positivity} (4.) and so the statement follows.
\item(5.) If $m=1$ and $n=k+1$, then
for $a=\frac{k}{2(k+1)}>\frac{n-1}{n(k+1)}$ 
\begin{align*}
D_k\left(\frac{k+2}{2(k+1)}\right)=(a+\frac{1}{k+1})(\psi_\sigma.B)+\frac{2a}{k}(\Delta_s.B)+(\psi_\tau.B)-(\Delta.B)=0
\end{align*}
by Proposition \ref{P:Positivity} (2.).
\end{proof}

Assembling this case-by-case analysis, we have
\begin{corollary}\label{C:LowerBound}
Let $\A:=\A_{n,m}^{k}$ be a weight vector with $k\geq 2$. Then there exists a rational number $c \leq \frac{k+2}{2(k+1)}$ such that $D_{k}(c)$ has non-negative intersection with any complete 1-parameter family of $\A$-stable curves with smooth general fiber. In addition, if $(n,m) \neq (k+1,1)$, we may take $c<\frac{k+2}{2(k+1)}$.
\end{corollary}
\begin{proof}
We need only check that thresholds defined in the four parts of Corollary \ref{C:Positivity}
each satisfy $c< \frac{k+2}{2(k+1)}$.
\begin{enumerate}
\item[1.] If $m=0$, then $n\geq 2k+1\geq k+3$, so we may take 
$c=\frac{n-1}{2(n-2)}\leq \frac{k+2}{2(k+1)}$ by Corollary \ref{C:Positivity} (1.) 
\item[2.] If $m=1$, then $n\geq k+1$. If $n=k+1$, then by Corollary \ref{C:Positivity} (5.) we can take $c=\frac{k+2}{2(k+1)}$.
If $n\geq k+2$, then we may take
$c=\frac{n+1}{2n}< \frac{k+2}{2(k+1)}$ by Corollary \ref{C:Positivity} (2.)  . 
\item[3.] 
If $m\geq 2$ and $n\geq k+1$, then we may take any
$1/2<c<\frac{n+1}{2n}\leq  \frac{k+2}{2(k+1)}$ by Corollary \ref{C:Positivity} (4.)
\item[4.] The case of $m\geq 2$ and $n\leq k$ is immediate from Corollary \ref{C:Positivity} (3.).
\end{enumerate}
\end{proof}

\section{Ample divisors on $\M_{0,\A}$}\label{S:Ampleness}
In this section, we will explain how to use the fundamental positivity result of Corollary \ref{C:LowerBound} to deduce the existence of certain ample divisors on $\M_{0,\A}$. As usual, we only consider weight vectors of the form $\A:=\A_{n,m}^{k}$, and when we say that a statement holds for all possible weight vectors, we mean all vectors of this form. For any weight vector $\A$ and $c \in \Q$, we have the divisor class
$$D_{k}(c):=c\psi_{\sigma}+(2c-1)\Delta_{s}+\psi_{\tau}-\Delta,$$
and our main result (Proposition \ref{P:Ampleness} and \ref{P:Ampleness2}) is
$$
D_{k}(c)\text{ is ample on }\M_{0,\A}\text{ if }c \in \left(\frac{k+2}{2k+2}, \frac{k+1}{2k}\right].
$$
The proof of the theorem breaks into several steps. In Section \ref{S:GeneralPositivity}, we prove that $D_{k}(c)$ has positive degree on every curve in $\M_{0,\A}$. In Section \ref{S:Perturbations}, we show that the same statement holds if we perturb $D_{k}(c)$ by a small linear combination of boundary divisors of $\M_{0,\A}$. Since the boundary divisors of $\M_{0,\A}$ generate its Picard group, this implies that $D_{k}(c)$ lies on the interior of the nef cone of $\M_{0,\A}.$ On smooth proper schemes, Kleiman's criterion implies that any divisor which lies on the interior of the nef cone is ample. Since Hassett has shown that $\M_{g,\A}$ is projective using Kollar's semipositivity techiques \cite{Hassett1}, we could stop here. In order to make our argument independent of Kollar's results, however, we explain in Section \ref{S:Kleiman} how to apply Kleiman's criterion without assuming \emph{a priori} that $\M_{0,\A}$ is projective. In general, Kleiman's criterion may fail for algebraic spaces (see \cite{Kollar2}, VI, 2.9.13). But it remains valid under certain hypotheses, and we can check these explicitly for $\M_{0,\A}$. Since it is often easier to construct moduli spaces as algebraic spaces rather than projective schemes, we wish to emphasize the point that it is actually possible to prove projectivity using our explicit intersection theory.

\subsection{Positivity on arbitrary 1-parameter families}\label{S:GeneralPositivity}
The following lemma allows us to pass from positivity results on families with smooth general fiber to positivity results on arbitrary families.
\begin{lemma}\label{L:BoundaryInduction}
Fix $a,b,c \in \Q$ and suppose that, for all weight vectors $\A$, the divisor $D:=a\psi_{\sigma}+b\Delta_{s}+c(\psi_{\tau}-\Delta)$ has non-negative (resp. \!\!positive) degree on any complete 1-parameter family of $\A$-stable curves with smooth general fiber. Then, for all weight vectors $\A$, $D$ has non-negative (resp.\!\! positive) degree on any complete 1-parameter family of $\A$-stable curves. In particular, $D$ is nef on each space $\M_{0,\A}.$
\end{lemma}
\begin{proof}
Simply observe that, if $\Delta_{S_1,S_2} \subset \M_{0,\A}$ is any boundary divisor, and
$$
\phi: \M_{0,\A_1} \times \M_{0,\A_2} \rightarrow \Delta_{S_1,S_2} \subset \M_{0,\A}
$$
is the natural gluing isomophism, then $D$ pulls back to the divisor of the same name on each factor, since
\begin{align*}
\phi^*(\psi_{\tau}-\Delta)&=\pi_1^{*}(\psi_{\tau}-\Delta) \times \pi_{2}^{*}(\psi_{\tau}-\Delta) \\
\phi^*(\psi_{\sigma})&=\pi_1^{*}(\psi_{\sigma}) \times \pi_{2}^{*}(\psi_{\sigma}) \\
\phi^*(\Delta_{s})&=\pi_1^{*}\Delta_{s} \times \pi_{2}^{*}\Delta_{s}.\\
\end{align*}

By induction, we have $\phi^*D=\pi_1^*D \times \ldots \times \pi_l^*D$ for an arbitrary boundary stratum
$
\phi: \prod_{j=1}^{l} \M_{0,\A_j} \rightarrow \M_{0,\A}.
$
Any curve $B \subset \M_{0,\A}$, whose general point passes through the interior of this stratum is numerically equivalent to $B_1+\ldots +B_l$, where $B_{j}$ lies in a fiber of 
$\prod_{j=1}^{l}\M_{0,\A_j} \rightarrow \prod_{j \neq i}\M_{0,\A_j},$
and the general point of $B_{i}$ lies in the interior of  $\M_{0,\A_i}$.
Thus, 
$$
D.B=\sum_{i=1}^{l}(\pi_i^*D).B_{i}=\sum_{i=1}^{l}D.\pi_i(B_i) >0.$$ 
Since every curve $B \subset \M_{0,\A}$ meets the interior of some boundary stratum, we are done.
\end{proof}

In order to obtain positivity statements for the divisors $D_{k}(c)$, we will now apply the positivity results of Section 3. Since the relevant statements for $\M_{0,n+m}$ are already known, we will record this case separately.

\begin{theorem}[Case $k=1$]\label{T:ClassicalCase}
Suppose that $n\geq 5$ or $m\geq 1$. For $c > 2/3$, the divisor $D_{1}(c)=c\psi_\sigma+\psi_\tau-\Delta$ is ample on $\M_{0,\A^1_{n,m}}=\M_{0,n+m}$.
\end{theorem}
\begin{proof}
When $c=1$, by 
Lemma \ref{L:PushforwardClass}, the divisor 
$D_1(1)$ is a positive rational multiple of the divisor $K_{\M_{0.n+m}}+\Delta$ which is ample by \cite[Lemma 3.6]{KeelMcKernan}. The general statement follows from Theorem 2.5.2 and Corollary 2.5.5 of Simpson's thesis \cite{Simpson}.

\end{proof}



\begin{theorem}[Main Theorem]\label{T:MainTheorem}
Let $\A:=\A_{n,m}^{k}$ be an arbitrary weight vector.
\begin{itemize}
\item[(a)] For $c \in (\frac{k+2}{2k+2}, \frac{k+1}{2k}]$, the divisor $D_{k}(c)$ has positive intersection with every curve in $\M_{0,\A}$.
\item[(b)] For $c = \frac{k+2}{2k+2}$, the divisor $D_{k}(c)$ is nef on $\M_{0,\A}$. Furthermore, it has degree zero precisely on those curves contracted by the reduction morphism
$
\M_{0,\A_{n,m}^{k}} \rightarrow \M_{0,\A_{n,m}^{k+1}}.
$
\end{itemize}
\end{theorem}

\begin{proof}
We proceed by induction on $k$. For the case $k=1$, we have Proposition \ref{T:ClassicalCase}, so we may assume that $k \geq 2.$

First, let us show that $D_{k}(c)$ is positive when $c=\frac{k+1}{2k}$.
Consider the reduction morphism 
$$\phi : \M_{0,\A_{n,m}^{k-1}} \rightarrow \M_{0,\A_{n,m}^{k}}.$$
Using Lemma \ref{L:Pullback}, one easily checks that
$$
\phi^*\left( D_{k}\left(\frac{k+1}{2k} \right) \right)=D_{k-1}\left( \frac{k+1}{2k} \right).
$$
Since $\frac{k+1}{2k}=\frac{(k-1)+2}{2(k-1)+2}$, the induction hypothesis implies that $D_{k-1}(\frac{k+1}{2k})$ is nef and has degree zero only on curves contracted by $\phi$. It follows that $D_{k}(\frac{k+1}{2k})$ has positive degree on all curves in $\M_{0,\A}$.

Now suppose that $(\C \rightarrow B, \{\sigma_i\}_{i=1}^{n}, \taum)$ is a generically smooth complete 1-parameter family of $\A$-stable curves, for some weight vector $\A$. By Corollary \ref{C:LowerBound}, there exists $c_{0} \leq \frac{k+2}{2k+2}$, such that $D_{k}(c_0).B \geq 0$. For any $c \in (\frac{k+2}{2k+2}, \frac{k+1}{2k}]$, the divisor $D_{k}(c)$ can be written as a convex linear combination
$$D_{k}(c)=\lambda D_{k}\left(\frac{k+1}{2k}\right)+ (1-\lambda)D_{k}(c_0), \,\,\text{ with $0<\lambda \leq 1$.}$$
Since we have already shown that $D_{k}(\frac{k+1}{2k}).B>0$, it follows that $D_{k}(c).B>0$ for all $c \in (\frac{k+2}{2k+2}, \frac{k+1}{2k}]$. By Lemma \ref{L:BoundaryInduction}, $D_{k}(c)$ has positive intersection on every curve in $\M_{0,\A}$. This completes the proof of part (a).

It remains to prove part (b). The fact that $D_{k}(\frac{k+2}{2k+2})$ is nef is immediate from part (a), since the nef cone is closed. Furthermore, if $(\C \rightarrow B, \{\sigma_i\}_{i=1}^{n}, \taum)$ is a generically smooth complete 1-parameter family of $\A$-stable curves, where $\A:=\A_{n,m}^{k}$ satisfies $(n,m) \neq (k+1,1)$, then Corollary \ref{C:LowerBound} gives $c_{0} < \frac{k+2}{2k+2}$, such that $D_{k}(c_0).B \geq 0$. Arguing as in the previous paragraph, we conclude that $D_{k}(\frac{k+2}{2k+2}).B>0.$ It follows that if  $(\C \rightarrow B, \{\sigma_i\}_{i=1}^{n})$ is a 1-parameter family on which $D_{k}(c)$ has degree zero, then every moving component of the generic fiber of $\C \rightarrow B$ must have $k+1$ marked points of weight $1/k$ and must be attached to the rest of the fiber in a single point. Equivalently, $B$ is contained in the fiber of the reduction morphism 
$
\M_{0,\A_{n,m}^{k}} \rightarrow \M_{0,\A_{n,m}^{k+1}}.
$
\end{proof}

\subsection{Perturbations of the fundamental divisor class $D_{k}(c)$}\label{S:Perturbations}

Theorem \ref{T:MainTheorem} nearly implies that $D_{k}(c)$ is ample on $\M_{0,\A}$, but of course one cannot always check ampleness simply by testing positivity on curves. In order to prove that $D_{k}(c)$ is ample, we will show that the divisor remains nef when perturbed by a small linear combination of boundary divisors. In fact, it is enough to consider perturbations by $S_{n} \times S_{m}$\nb-equivariant divisors, as we explain in Lemma \ref{L:PerturbationLemma}. The same methods used in the proof of Theorem \ref{T:MainTheorem} are easily adapted to prove this stronger statement. Thoughout this section, we write $\Delta_{i,j} \subset \M_{0,\A}$ to denote the sum of all irreducible components of the boundary of the form $\Delta_{S_1, S_2}$ where $S_1$ is a subset of $i$ weight $1/k$ sections and $j$ weight $1$ sections of the universal curve. Note that $\Delta_{i,j}$ is invariant under the action 
of $S_{n} \times S_{m}$ by definition.

\begin{lemma}\label{L:PerturbationLemma}
Suppose that, for a weight vector $\A$ and $c_0 \in (\frac{k+2}{2k+2}, \frac{k+1}{2k})$, there exists  
$\epsilon=\epsilon(k,n,m,c_0)>0$ such that for all $\epsilon_{i,j} \in \Q \cap [-\epsilon,\epsilon]$ and 
all $c\in \Q\cap [c_0-\epsilon,c_0+\epsilon]$, the divisor
 $$D_{k}(c)+\sum_{i,j}\epsilon_{i,j}\Delta_{i,j}=D_k(c_0)+(c-c_0)(\psi_\sigma+2\Delta_s)+\sum_{i,j}\epsilon_{i,j}\Delta_{i,j}$$ has positive intersection on any complete 1-parameter family of $\A$-stable curves. Then $D_{k}(c_0)$ is ample on $\M_{0,\A}$.
 \end{lemma}
\begin{proof}
Let $\pi:\M_{0,\A} \rightarrow \M_{0,\A}/S_{n} \times S_{m}$ be the quotient morphism for the natural action of $S_{n} \times S_{m}$ on $\M_{0,\A}$. Since $D_{k}(c_0)$ is $S_{n} \times S_{m}$-equivariant, we have
$$
D_{k}(c_0)=\pi^*D_{k}'(c_0)
$$ 
for some divisor class $D_{k}'(c_0)$ on $\M_{0,\A}/S_{n} \times S_{m}$, and it suffices to prove that $D_{k}'(c_0)$ is ample. Since $\psi_\sigma+2\Delta_s$ and the boundary divisors $\Delta_{i,j}$ of $\M_{0,\A}$ generate $\Pic_{\Q}(\M_{0,\A})$, we have that $\Pic_{\Q}(\M_{0,\A}/S_{n} \times S_{m})$ is generated by the images of these equivariant divisors, i.e. the images of $\psi_\sigma+2\Delta_s$ and the boundary divisors $\Delta_{i,j}$. 

By assumption, $D_{k}'(c_0)$ lies on the interior of the nef cone of $\M_{0,\A}/S_{n} \times S_{m}.$ Since $\M_{0,\A}$ is smooth projective, $\M_{0,\A}/S_{n} \times S_{m}$ is projective with $\Q$-factorial singularities, and we may apply Kleiman's criterion to conclude that $D_{k}'(c)$ is ample.

\end{proof}

In order to apply Lemma \ref{L:PerturbationLemma}, we must check that the statements of Proposition \ref{P:Positivity} remain valid when we replace $D_{k}(c)$ by a small perturbation $D_{k}(c)+
\sum_{i,j}\epsilon_{i,j}\Delta_{i,j}$.
\begin{proposition}\label{P:Positivity-perturbed} For any weight vector $\A$ and rational number $\delta>0$, there exists $\epsilon=\epsilon(\delta,k,n,m)$ such that, for any generically smooth 1-parameter family of $\A$-stable curves $(\C \rightarrow B, \sigman, \taum)$ and any $\epsilon_{i} \in \Q \cap [-\epsilon,\epsilon]$, the following inequalities are satisfied.
 \begin{itemize}
\item[1.] If $m=0$, then for $a > \frac{n-1}{(n-k-1)(k+1)}+\delta$ we have
\begin{align*}
a(\psi_{\sigma}.B)+\frac{2a}{n-1}(\Delta_{s}.B)-(\Delta.B)+\sum_{i,j}\epsilon_{i,j}(\Delta_{i,j}.B) > 0.
\end{align*}
\item[2.] If $m=1$ and $n\geq k+2$, then for $a > \frac{n-1}{n(k+1)}+\delta$ we have
\begin{align*}
(a+\frac{1}{n})(\psi_{\sigma}.B)+\frac{2a}{n-1}(\Delta_{s}.B)+(\psi_{\tau}.B)-(\Delta.B) +\sum_{i,j}\epsilon_{i,j}(\Delta_{i,j}.B)> 0.
\end{align*}
\item[3.] If $m\geq 2,\, 2 \leq n \leq k$, then for $a > \delta$ and $b >\delta$ we have
\begin{align*}
(a+\frac{b}{n})(\psi_{\sigma}.B)+\frac{2a}{n-1}(\Delta_{s}.B)+(\psi_{\tau}.B)-(\Delta.B)+\sum_{i,j}\epsilon_{i,j}(\Delta_{i,j}.B) > 0.
\end{align*}
\item[4.] If $m\geq 2,\, n \geq k+1$, then for  $\frac{(k+1)(n-k-1)}{n-1}a+\frac{k+1}{n}b>1+\delta \text{ and }b>1+\delta$ we have
\begin{align*}
(a+\frac{b}{n})(\psi_{\sigma}.B)+\frac{2a}{n-1}(\Delta_{s}.B)+(\psi_{\tau}.B)-(\Delta.B)+\sum_{i,j}\epsilon_{i,j}(\Delta_{i,j}.B)> 0.
\end{align*}
\end{itemize}
\end{proposition}
\begin{proof}
The proof is essentially identical to the proof of Proposition \ref{P:Positivity}, so we provide details only for the case $m=0$.
We consider a sequence of birational contractions $\C_0\ra \C_1\ra \dots \ra \C_N$, where $\C_0$ is a minimal desingularization of $\C$ and 
$\C_N$ is a $\P^1$\nb-bundle over $B$. We then have
have 
$$a(\psi_{\sigma}.B)+\frac{2a}{n-1}(\Delta_{s}.B)-(\Delta.B)+\sum_{i,j}\epsilon_{i,j}(\Delta_{i,j}.B)=aF_\sigma(0)-F_\Delta(0)+\sum_{i,j}\epsilon_{i,j}F_{i,j}(0),$$
and it suffices to show that the function $G(r):=aF_\sigma(r)-F_\Delta(r)+\sum_{i,j} \epsilon_{i,j}F_{i,j}(r)$ is a decreasing function of $r$. Here $F_{i,j}:[0,N] \rightarrow \mathbb{Z}$ is the function defined by setting $F_{i,j}(r)$ equal to the number of disconnecting nodes in fibers of the family $\C_r\ra B$ which separate $i$ sections of weight $1/k$ and $j$ sections of weight $1$ from the rest of the sections.
Suppose that the exceptional divisor of $\C_r\ra \C_{r+1}$ meets $r_1\geq k+1$ sections so that we are eliminating a node corresponding to the boundary component $\Delta_{r_1,0}$. Then, using Lemma 3.1 (b), we have for $0<\epsilon<\frac{(n-k-1)(k+1)}{n-1}\delta$:
$$G(r)-G(r+1)=\frac{(r_1+1)(n-r_1)}{n-1}a-1+\epsilon_{r_1,0}>0$$ 
for $a>\frac{(n-1)}{(n-k-1)(k+1)}+\delta$.
\end{proof}

\begin{corollary}\label{C:Positivity-perturbed}
Let $\A$ be a weight vector with $k\geq 2$. There exists $\epsilon:=\epsilon(k,n,m)>0$ and $c_0 \leq \frac{k+2}{2(k+1)}$ such that, for any generically smooth 1-parameter family of 
$\A$-stable curves $(\C\rightarrow B, \{\sigma_j\}_{j=1}^n, \{\tau_j\}_{j=1}^m)$ and any $\epsilon_{i,j} \in \Q \cap [-\epsilon, \epsilon]$,
we have 
$$D_{k}(c_0).B+\sum_{i,j}\epsilon_{i,j}(\Delta_{i,j}.B)>0.$$
\end{corollary}
\begin{proof}
The proof is a slight modification of the proof of Corollary \ref{C:Positivity}. Therefore, we will only
provide details for the case $m=0$. 

Suppose $m=0$. For $\delta$ sufficiently small, 
$c_0=\frac{n-1}{2(n-2)}>\frac{n-1}{(n-k-1)(k+1)}+\delta$ and so by  Proposition \ref{P:Positivity-perturbed} (1.) there exist $\epsilon=\epsilon(\delta,k,n,m)$ such that
\begin{align*}
&(D_k(c_0)+\sum_{i,j}\epsilon_{i,j}\Delta_{i,j}).B =
c_0(\psi_{\sigma}.B)+\frac{2c_0}{n-1}(\Delta_{s}.B)-(\Delta.B)+\sum_{i,j}\epsilon_{i,j}(\Delta_{i,j}.B) > 0,
\end{align*}
for all $|\epsilon_{i,j}|\leq \epsilon$.
It remains to observe that $c_0\leq \frac{k+2}{2(k+1)}$.

\end{proof}

Following the arguments of Theorem \ref{T:MainTheorem}, we obtain

\begin{proposition}\label{P:Ampleness} For $c \in (\frac{k+2}{2k+2}, \frac{k+1}{2k})$, the divisor $D_{k}(c)$ is ample on $\M_{0,\A}$.
\end{proposition}
\begin{remark}
The divisor $D_k(c)$ is numerically proportional to the divisor class
$\psi-2\Delta+\alpha(\Delta+\Delta_s)$ for $\alpha=2-1/c$ by Lemma \ref{L:PushforwardClass}. Therefore 
this proposition together with Proposition \ref{P:Ampleness2} establish the Main Result  of Section 1 
and its Corollary, thus providing an unconditional proof of Simpson's Theorem \ref{T:Simpson}.

\end{remark}
\begin{proof}
We fix $k$ and proceed by induction on $\dim \M_{0,\A}$. The case of $\dim \M_{0,\A}=0$ is trivial. Suppose the statement is established 
for all weight vectors $\A'=\A^k_{n,m}$ with $\dim \M_{0,\A'}<\dim \M_{0,\A}$. 

By Lemma \ref{L:PerturbationLemma}, it suffices to show that there exists an $\epsilon>0$ such that $D_{k}(c')+\sum_{i,j}\epsilon_{i,j}\Delta_{i,j}$ has non-negative intersection on any 1-parameter family of $\A$-stable curves for all $\epsilon_{i,j} \in [-\epsilon, \epsilon] \cap \Q$ and all
$c'\in \Q\cap [c-\epsilon, c+\epsilon]$. 

By the induction assumption and a slight 
variation of the proof of Lemma \ref{L:BoundaryInduction}, $D_{k}(c')+\sum_{i,j}\epsilon_{i,j}\Delta_{i,j}$ has non-negative degree on any complete 1-parameter family with reducible generic fiber. It remains 
to show that $D_k(c')+\sum_{i,j}\epsilon_{i,j}\Delta_{i,j}$ has non-negative degree on any complete 
1-parameter family with smooth general fiber. 

By Corollary \ref{C:Positivity-perturbed}, there exists $c_{0} \leq \frac{k+2}{2k+2}$ and $\epsilon'>0$ such that $D_{k}(c_0)+\epsilon_{i,j}'\Delta_{i,j}$ has non-negative intersection on any generically smooth 1-parameter family of $\A$-stable curves for all $\epsilon_{i,j}' \in [-\epsilon',\epsilon'] \cap \Q$. By Theorem \ref{T:MainTheorem}, $D_{k}(\frac{k+1}{2k})$ has positive intersection on any generically smooth 1-parameter family of $\A$-stable curves . For any $c' \in [c-\epsilon, c+\epsilon]\subset (\frac{k+2}{2k+2}, \frac{k+1}{2k})$, there exists $0<\lambda<1$ such that
$$
D_{k}(c')=\lambda (D_{k}(c_0))+(1-\lambda)D_{k}\left(\frac{k+1}{2k}\right).
$$
Clearly, all such $\lambda$ are bounded below by some positive number $\lambda_0$.
Furthermore,  for any $\epsilon_{i,j} \in [-\epsilon'\lambda, \epsilon'\lambda]$, we can write
$$
D_{k}(c')+\sum_{i,j}\epsilon_{i,j}\Delta_{i,j}=\lambda (D_{k}(c_0)+\sum_{i,j}\epsilon_{i,j}'\Delta_{i,j})+(1-\lambda)D_{k}\left(\frac{k+1}{2k}\right),
$$
where $\epsilon_{i,j}' \in [-\epsilon', \epsilon'] \cap \Q$. It follows that $D_{k}(c')+\sum_{i,j}\epsilon_{i,j}\Delta_{i,j}$ has non-negative intersection on any generically smooth 1-parameter family of $\A$-stable curves. Using $\epsilon=\lambda_0 \cdot \epsilon'$ gives the desired result.
\end{proof}

It remains to check ampleness at the endpoint $c=\frac{k+1}{2k}$.
\begin{proposition}\label{P:Ampleness2}
For any weight vector $\A$, $D_k(\frac{k+1}{2k})$ is ample on $\M_{0,\A}$.
\end{proposition}
\begin{proof}

We proceed by induction on $\dim \M_{0,\A}$. Fix a weight vector $\A$, and assume that the given statement holds for all weight vectors $\A'$ satisfying $\dim \M_{0,\A'}<\dim \M_{0,\A}$. By Proposition \ref{P:Ampleness}, divisors $D_k(c)$ is ample for 
$c\in \left(\frac{k+2}{2k+2},\frac{k+1}{2k}\right)$. To show that $D_{k}(\frac{k+1}{2k})$ is ample, it suffices to exhibit a rational number $\epsilon>0$ such that $D_k(\frac{k+1}{2k}+\epsilon)$ is nef.

We will show that for small enough $\epsilon$ and any compete curve $B \subset \M_{0,\A}$, we have $D_k(\frac{k+1}{2k}+\epsilon).B \geq 0$. If
$$ \phi:\M_{0,\A_{n,m}^{k-1}} \rightarrow \M_{0,\A_{n,m}^{k}}$$
is the natural reduction morphism, we will consider separately the cases where $B \subset \phi(\Exc(\phi))$ and $B \not\subset \phi(\Exc(\phi)).$ Suppose first that $B \subset \phi(\Exc(\phi))$. In this case, $k$ of the sections $\sigman$ are coincident on the corresponding family of $\A$-stable curves. It follows that $B$ lies in the image of  the closed immersion $\chi: \M_{0,\A_{n-k,m+1}^k}\ra \M_{0,\A_{n,m}^k}$, which replaces the $(m+1)^{\text{st}}$ section of weight one with the $k$ coincident sections of weight $1/k$. By Lemma \ref{L:PullbackX}, we have

$$\chi^*D_k\left(\frac{k+1}{2k}+\epsilon\right)=D_k\left(\frac{k+1}{2k}\right)-\epsilon k(k-2) \psi_{\tau_{m+1}}.$$
Since  $\dim \M_{0,\A_{n-k,m+1}^k}<\dim \M_{0,\A}$, the induction hypothesis implies that 
$D_k(\frac{k+1}{2k})$ is ample on $\M_{0,\A_{n-k,m+1}^k}$. There exists $\epsilon>0$ sufficiently small so that $D_k(\frac{k+1}{2k})-\epsilon k(k-2) \psi_{\tau_{m+1}}$ is still ample on $\M_{0,\A_{n-k,m+1}^k}$, and for this choice of $\epsilon$, we have $D_k\left(\frac{k+1}{2k}+\epsilon\right).B \geq 0$ as desired.

Next, suppose that $B \not\subset \phi(\Exc(\phi)).$  Let $B'$ denote the $\phi$-transform of $B$ on $\M_{0,\A_{n,m}^{k-1}}$ so that
$$
D_k\left(\frac{k+1}{2k}+\epsilon\right).B=\phi^*D_k\left(\frac{k+1}{2k}+\epsilon\right).B'.
$$
By Lemma \ref{L:Pullback},  we have
$$\phi^*D_k\left(\frac{k+1}{2k}+\epsilon\right)=D_{k-1}\left(\frac{k+1}{2k}+\epsilon\right)+\epsilon k(k-2)E,$$
where $E$ is the union of the exceptional divisors of $\phi$. Since $\frac{k+1}{2k}+\epsilon=\frac{(k-1)+2}{2(k-1)+2}+\epsilon$, Proposition \ref{P:Ampleness} implies that for sufficiently small $\epsilon$,
$D_{k-1}(\frac{k+1}{2k}+\epsilon).B'>0$. Since $B'$ is not contained in $E$, 
we also have $E.B'>0$. It follows that $D_k\left(\frac{k+1}{2k}+\epsilon\right).B \geq 0$, as desired.
\end{proof}

\subsection{Kleiman's criterion on an algebraic space}\label{S:Kleiman}
In the proof of Proposition \ref{P:Ampleness}, we used the fact that $\M_{0,\A}$ (or rather the quotient $\M_{0,\A}/S_{n} \times S_{m}$) is a scheme when we invoked Kleiman's criterion. Since one often encounters situations where one would like to prove projectivity of a moduli space without knowing \emph{a priori} that it is a scheme, it seems worth pointing out that our method can be used to prove projectivity by means of the following lemma.

The reader may consult (\cite{Kollar2}, VI.2) for a basic treatment of intersection theory which is applicable on algebraic spaces. The proof of the following lemma is nothing more than a logical rehashing of the proof that Nakai's criterion (\cite{Kollar2}, VI, 2.18) implies Kleiman's criterion (\cite{Kollar2}, VI, 2.19).
\begin{lemma}[Kleiman's criterion on algebraic spaces]\label{L:Kleiman}
Suppose that $X$ is a algebraic space, proper over an algebraically closed field. Suppose $X$ has the property that, for any subvariety $Z \subset X$, there exists an effective Cartier divisor $E$ such that $E$ meets $Z$ properly. Then Kleiman's criterion holds for $X$, i.e. any divisor $D$ which lies in the interior of the nef cone of $X$ is ample.
\end{lemma}
\begin{proof}
Suppose that $D \subset X$ is a Cartier divisor which lies in the interior of the nef cone of $X$. To prove that $D$ is ample, it suffices to show that $D^{k}.[Z] >0$ for an arbitrary $k$-dimensional subvariety $Z \subset X$.

Given a $k$-dimensional subvariety $Z \subset X$, our hypothesis gives an effective Cartier divisor $E \subset X$, such that
$$
E.[Z]=\sum_{i}a_i[Z_i],
$$
with each $Z_{i} \subset X$ a $(k-1)$-dimensional subvariety, and each $a_{i}>0$. Since $D$ lies on the interior of the nef cone, there exists a rational number $\epsilon>0$ such that $D-\epsilon E$  is nef. Now we have
 $$D^{k-1}(D-\epsilon E).[Z] \geq 0,$$
since $D$ and $D-\epsilon E$ are nef (\cite{Kollar2}, VI, 2.18.7.3).  It follows that
 $$
 D^{k}.[Z] \geq \epsilon D^{k-1}[E. Z]=\epsilon \sum_{i}a_iE.[Z_i]> 0,
 $$ by induction on the dimension of $Z$.
\end{proof}

In order to prove that Kleiman's criterion holds for $\M_{0,\A}$ without knowing a priori that $\M_{0,\A}$ is a scheme, it suffices to check that $\M_{0,\A}$ satisfies the hypothesis of Lemma \ref{L:Kleiman}. (Technically, we applied Kleiman's criterion to the quotient $\M_{0,\A}/S_{n} \times S_{m}$, but it is clear that if $\M_{0,\A}$ satisfies the hypothesis of Lemma \ref{L:Kleiman}, then so does $\M_{0,\A}/S_{n} \times S_{m}$.)

If  $Z \subset \M_{0,\A}$ is an arbitrary positive-dimensional subvariety, the general point of $Z$ lies in the interior of some boundary stratum
$$
\M_{0,\A_1} \times \ldots \times \M_{0,\A_k} \subset \M_{0,\A}, 
$$

Since the interior $M_{0,n_1} \times \ldots \times M_{0,n_k} \subset \M_{0,\A_1} \times \ldots \times \M_{0,\A_k}$ is affine, $Z$ must meet some irreducible component of the boundary of $\M_{0,\A_1} \times \ldots \times \M_{0,\A_k}$. Since every  irreducible component of the boundary of $\M_{0,\A_1} \times \ldots \times \M_{0,\A_k}$ is the restriction of a boundary divisor on the ambient space $\M_{0,\A}$, all of which are Cartier, we are done.

\end{document}